\documentclass{amsart}
\usepackage{amsmath}
  \usepackage{paralist}
  \usepackage{graphics} 
  \usepackage{epsfig} 
\usepackage{graphicx}  \usepackage{epstopdf}
 \usepackage[colorlinks=true]{hyperref}
\hypersetup{urlcolor=blue, citecolor=red}

\usepackage{amssymb}
\usepackage{latexsym}
\usepackage{amsthm}
\usepackage{amscd}
\usepackage{color}

\newtheorem{thm}{Theorem}[section]

\newtheorem{lem}[thm]{Lemma}
\newtheorem{prop}[thm]{Proposition}

\numberwithin{equation}{section}

\newcommand{\wt}{\widetilde}

\newcommand{\al}{\alpha}
\newcommand{\Ga}{\Gamma}
\newcommand{\ep}{\varepsilon}

\allowdisplaybreaks[1]

\begin{document}
\title[cubic fractional  Schr\"odinger equation]{Well-posedness and Ill-posedness \\ for {the}
cubic fractional Schr\"odinger equations}
\author[Y. Cho]{Yonggeun Cho}
\address{Department of Mathematics, and Institute of Pure and Applied Mathematics, Chonbuk National University, Jeonju 561-756, Republic of Korea}
\email{changocho@jbnu.ac.kr}
\author[G. Hwang]{Gyeongha Hwang}
\address{Department of Mathematical Sciences, Ulsan National Institute of Science and Technology, Ulsan, 689-798, Republic of Korea}
\email{ghhwang@unist.ac.kr}
\author[S. Kwon]{Soonsik Kwon}
\address{Department of Mathematical Sciences, Korea Advanced Institute of Science and Technology, Daejeon 305-701, Republic of Korea} \email{soonsikk@kaist.edu}
\author[S. Lee]{Sanghyuk Lee}
\address{Department of Mathematical Sciences, Seoul National University, Seoul 151-747, Republic of Korea}
\email{shklee@snu.ac.kr}


\thanks{}

\begin{abstract}
We study the low regularity well-posedness of the 1-dimensional
cubic nonlinear fractional Schr\"odinger equations with L\'{e}vy
indices  $1 < \al < 2$. We consider both non-periodic and periodic
cases, and prove that the Cauchy problems are locally well-posed in
$H^s$ for $s \geq \frac {2-\al}4$. This is shown via a trilinear
estimate in Bourgain's $X^{s,b}$ space. We also show that
non-periodic equations are ill-posed  in $H^s$ for $\frac {2 -
3\al}{4(\al + 1)} < s < \frac {2-\al}4$ in the sense that the flow
map is not locally uniformly continuous.
\end{abstract}

\maketitle
\section{Introduction}

We consider the Cauchy problem for {the} one dimensional fractional
Schr\"odinger equations with cubic nonlinearity in periodic  and non
periodic settings:
\begin{align}\label{eq}
\left\{\begin{array}{l}
i\partial_tu + (-\Delta)^{\al/2}u = \gamma|u|^2u,\\
u(0,\cdot) = \phi \in H^s(\widehat Z),
\end{array}
\right.
\end{align}
where $\widehat Z = \mathbb R$ or $\mathbb T$, $\alpha\in (1,2)$ is
the L\'{e}vy index, $\gamma \in \mathbb R \setminus \{0\}$ and $s
\in \mathbb R$. In this paper we are concerned with well-posedness
of the Cauchy problem in low regularity Sobolev spaces. As the
linear part generalizes the usual second-order Schr\"odinger
equation, our interest  is to investigate how the weaker dispersion
affects dynamics and well-posedness.
 The fractional Schr\"odinger
equations was introduced in {the} theory of the fractional quantum
mechanics where the Feynmann path integrals approach is generalized
to $\alpha$-stable L\'evy process \cite{la1}. Also it appears in the
water wave models (for example, see \cite{iopu} and references
therein).

In what follows $Z$ denotes $ \mathbb R$ (non-periodic) or $\mathbb
Z$ (periodic). Accordingly, the Sobolev space $H^s(\widehat Z)$ is
defined by
$$
H^s(\widehat Z) = \big\{f \in \mathcal S' : \|f\|_{H^s(\widehat Z)} := \|(1 + |\xi|^2)^\frac s2 \mathcal{F} f \,\|_{L^2(Z)} < \infty\big\},
$$
where $L^2( Z)$ denotes $L^2(\mathbb R)$ or  $\ell^2(\mathbb Z)$ and
$\mathcal{F} f$ is {the} Fourier transform or Fourier coefficient of
$f$ given by $\mathcal{F} f(\xi) = \int_{\widehat Z} e^{- i x \xi} f
dx$ for $\xi \in Z$.

We {define the linear propagator $U(t)$ by setting}
\[U(t)\phi = e^{i(-\Delta)^{\al/2}t}\phi = \mathcal{F}^{-1} e^{i|\xi|^{\al}t} \mathcal{F}\phi,\]
where $\mathcal F^{-1}$ denotes the inverse Fourier transform. Then,
by {Duhamel's} formula the equation \eqref{eq} is written as an
integral equation \begin{equation}\label{integral} u = U(t)\phi -
i\gamma \int_0^t U(t-t')(|u|^2u(t'))dt'.\end{equation}

\subsubsection*{Well-posedness} If $s > 1/2$, by the Sobolev embedding and the energy
method one can easily show the local well-posedness in $H^s$ for $0
< \al < 2$ for both periodic and non periodic cases. The equation
\eqref{eq} also has the mass and energy conservation:
$$  M(u)= \int |u|^2, \qquad E(u) = \frac 12 \int | |\nabla|^{\al/2}u|^2  + \gamma\frac 14 \int |u|^4. $$
Thus, for $s \ge {\al}/2$ and $s > 1/2$, the global well-posedness
in $H^s$ follows from the conservation laws. (For instance see
\cite{chho, chho2}.)

For the less regular initial data, i.e. $s\le 1/2$,  particularly in
the non periodic case, {a plausible approach may be to use the
Strichartz estimate for $U(t)$.} In fact, it is known that the
estimate
\begin{align}\label{str}
\||\nabla|^{-\frac{2-\alpha}{q}}U(t)\phi\|_{L^q_tL^r_x(\mathbb R
\times \mathbb R)} \lesssim \|\phi\|_{L^2_x}\end{align} holds for
$2/q + 1/r = 1/2,\; 2\le q, r \le \infty$ (see \cite{cox}). However,
due to weak dispersion the estimate accompanies a derivative loss of
order $2$-$\alpha$ unless one imposes additional assumptions on
$\phi$ (\cite{chkl1,cholee}). This makes difficult for general data
to use the usual iteration argument which relies on \eqref{str}.

To get around the shortcoming of Strichartz estimates we use
Bourgain's $X^{s,b}$ space, which  has been widely used in the
studies of dispersive equations for both non periodic and periodic
setting. For the fractional Schr\"odinger equation, $X_{\widehat
Z}^{s,b}$ is defined by
\[X_{\widehat Z}^{s,b} = \big\{\varphi \in \mathcal S' : \|\varphi\|_{X_{\widehat Z}^{s,b}} := \|\langle \xi \rangle^s \langle \tau - |\xi|^\al \rangle^b \widehat \varphi (\tau, \xi)
\|_{L^2(\mathbb R \times Z)} < \infty\big\}, \] where $\widehat
\varphi (\tau, \xi)$ is the Fourier transform of $\varphi$ with
respect to the time and space variables. Here $\langle \cdot
\rangle$ denotes $1 + |\cdot|$. For the standard iteration argument,
the main step is to show the trilinear estimate in terms of
$X^{s,b}$ spaces:
\begin{equation}\label{trili}
  \| uvw\|_{X^{s,b-1}_{\widehat Z}} \lesssim \|u\|_{X^{s,b}_{\widehat Z}}\|v\|_{X^{s,b}_{\widehat Z}}
\|w\|_{X^{s,b}_{\widehat Z}}. \end{equation} We obtain  this
estimate by adapting the dyadic method in Tao \cite{tao} in which
multilinear estimates in weighted $L^2$ spaces are systematically
studied. The argument similarly applies to both non periodic and
periodic cases.

The following is our local well-posedness result.



\begin{thm}\label{main1}
For $1 < \al <2$, the Cauchy problem \eqref{eq} is locally well-posed in $H^s(\widehat Z)$, if $s \geq \frac {2 - \al}4$.
\end{thm}

Recently, for the periodic case,  Demirbas,  Erdo\u{g}an and
Tzirakis \cite{det} showed that the equation \eqref{eq} is locally
well-posed for $s > \frac {2-\al}{4}$ and globally wellposed for $s
> \frac {5\al + 1}{12}$. Our result gives local well-posedness at the missing endpoint $s=\frac
{2-\al}{4}$.

The regularity threshold $ s= \frac{2-\al}{4}$ is optimal in that
below that number we do not expect to solve \eqref{eq} via the
contraction mapping principle. Firstly, the estimate \eqref{trili}
fails for $s<\frac{2-\al}{4} $ due to the resonant interaction of
high--high--high to high (frequencies). Compared to the usual
Schr\"odinger equation, the curvature of the characteristic curve is
smaller ($ (\text{frequency})^{\al-2} $). So, the stronger such
resonant interactions make the threshold regularity higher. See the
counter-example in Section 4. In \cite{gh}, the authors claimed that
\eqref{eq} is globally well posed if $\phi\in L^2$. But Theorem
\ref{ill} below shows that their result is incorrect. Their proof is
based on a trilinear estimate, namely \eqref{tri} with $s=0$
(\cite[Theorem 3.2]{gh}), which is not true.

\smallskip

\subsubsection*{Ill-posedness} Now we consider ill-posedness in the non periodic
setting. Following Christ, Colliander, and Tao \cite{cct}, we
approximate the fractional equations with the cubic NLS, at
$(N,N^\al)$ in the Fourier space by Taylor expansion of the phase
function. This allows to transfer an ill-posedness result of NLS to
\eqref{eq}. A similar trick was also used in the fifth-order
modified KdV equation \cite{kwon}. The following is our second
result.

\begin{thm}\label{ill}
Let $\frac {2 - 3\al}{4(\al + 1)} < s <  \frac {2 - \al}4$. Then the
solution map of the initial value problem \eqref{eq} fails to be
locally uniformly continuous on $C_TH^s(\mathbb R)$ for any $T>0$.
More precisely, for $0 < \delta \ll \ep \ll 1$ and $T > 0$
arbitrary, there are two solutions $u_1, u_2$ to \eqref{eq} with
initial data $\phi_1, \phi_2$ such that
\begin{align}\label{ill-1}
&\|\phi_1\|_{H^s},\|\phi_2\|_{H^s} \lesssim \ep\,,
\\
\label{ill-2} &\|\phi_1 - \phi_2\|_{H^s} \lesssim \delta\,,
\\
\label{ill-3} \sup_{0 \leq t \leq T}&\|u_1(t) - u_2(t)\|_{H^s}
\gtrsim \ep.
\end{align}
\end{thm}

In view of the counter-example of the trilinear estimate
\eqref{trili} it seems natural to expect the similar ill-posedness
result for the periodic equations. However, it is not so simple to
set make up a counter example because  the frequency supports are
distributed in a wide region of length $N^{\frac{2-\al}{2}}$.
Currently we are not able to prove ill-posedness\,\footnote{ The
same counter-example as the cubic NLS in \cite{bst} gives the
ill-posedness for $ s<0$.  }.

\subsubsection*{Organization of the paper} The paper is organized as follows. In section 2, we
introduce notations and recall previously known estimates which we
need in the subsequent section. In section 3, bilinear estimates in
$X_{\widehat Z}^{s,b}$ space  are established. Finally, we prove
Theorem \ref{main1} in section 4 and Theorem \ref{ill} in section 5.

\section{Notations and Preliminaries}
We will use the same notations as in \cite{tao}. Let us invoke that
$Z$ denotes $\mathbb R$ for the non-periodic case and $\mathbb Z$
for the periodic case. For any integer $k \ge 2$, let $\Ga_k(\mathbb
R \times Z)$ denote the hyperplane
\[\Ga_k(\mathbb R \times Z) := \{\zeta = (\zeta_1, \cdots, \zeta_k) \in (\mathbb R \times Z)^{k} : \zeta_1 + \cdots + \zeta_k = 0\}
\]
with \[\int_{\Ga_k(\mathbb R \times Z)} f:= \int_{(\mathbb R \times
Z)^{k-1}} f(\zeta_1, \cdots, \zeta_{k-1}, -\zeta_1 - \cdots -
\zeta_{k-1}) d\zeta_1 \cdots d\zeta_{k-1},\] where $d\zeta_j$ is the
product of Lebesgue and the counting measure for the periodic case,
and the Lebesgue measure on $\mathbb R^2$ for the non-periodic case.
Note that the integral is symmetric under permutations of $\zeta_j$.

Let us define a $[k; \mathbb R \times Z]$-multiplier to be any
function $m : \Ga_k(\mathbb R \times Z) \rightarrow \mathbb{C}$.
When $m$ is a $[k; \mathbb R \times Z]$-multiplier, the norm
$\|m\|_{[k;\mathbb R \times Z]}$ is defined to be  the best constant
so that the inequality
\[|\int_{\Ga_k(\mathbb R \times Z)} m(\zeta) \prod_{j=1}^k f_j(\zeta_j)| \le \|m\|_{[k;\mathbb R \times Z]} \prod_{j=1}^k \|f_j\|_{L^2(\mathbb R \times Z)}\]
holds for all test functions $f_j$ on $\mathbb R \times Z$.  Here we
recall some of the results about $[k; \mathbb R \times
Z]$-multiplier from \cite{tao},  which is to be used later.
\begin{lem}\label{compa}
If $m$ and $M$ are $[k; \mathbb R \times Z]$-multipliers, and
$|m(\zeta)| \leq M(\zeta)$ for all $\zeta \in \Ga_k(\mathbb R \times
Z)$, then $\|m\|_{[k; \mathbb R \times Z]} \leq \|M\|_{[k; \mathbb R
\times Z]}$. Also, if $m$ is a $[k; \mathbb R \times Z]$-multiplier,
and $g_1, \cdots, g_k$ are functions from $\mathbb R \times Z$ to
$\mathbb{R}$, then
\[\big\|m(\zeta)\prod_{j=1}^k g_j(\zeta_j)\big\|_{[k; \mathbb R \times Z]} \le \|m\|_{[k; \mathbb R \times Z]} \prod_{j=1}^k \|g_j\|_{\infty}.\]
\end{lem}
\begin{lem}\label{trans}
For $\zeta_0 \in \Ga_k(\mathbb R \times Z)$ and a $[k; \mathbb R
\times Z]$-multiplier $m$, we have
\[\|m(\zeta)\|_{[k; \mathbb R \times Z]} = \|m(\zeta + \zeta_0)\|_{[k; \mathbb R \times Z]}.\]
From this and Minkowski's inequality, we thus have the averaging
estimate, for any finite measure $\mu$ on $\Ga_k(\mathbb R \times
Z)$,
\[\|m*\mu\|_{[k; \mathbb R \times Z]} \leq \|m\|_{[k; \mathbb R \times Z]}\|\mu\|_{L^1(\Ga_k(\mathbb R \times Z))}.\]
\end{lem}
\begin{lem}\label{compo}
Let $k_1, k_2 \ge 1$, and $m_1, m_2$ be functions defined on
$(\mathbb R \times Z)^{k_1}$, $(\mathbb R \times Z)^{k_2}$,
respectively. Then
\begin{align*}
\|m_1(\zeta_1, \cdots,\zeta_{k_1})m_2(\zeta_{k_1+1},\cdots,\zeta_{k_1+k_2})\|_{[k_1 + k_2;\mathbb R \times Z]} \leq \|m_1\|_{[k_1+1;\mathbb R \times Z]}\|m_2\|_{[k_2+1;\mathbb R \times Z]}.
\end{align*}
As a special case, we have the $TT^*$ identity, for all functions
$m:(\mathbb R \times Z)^{k} \rightarrow \mathbb{R}$,
\[\|m(\zeta_1, \cdots,\zeta_{k})\overline{m(-\zeta_{k+1},\cdots,-\zeta_{2k})}\|_{[2k;\mathbb R \times Z]}\\
\leq \|m(\zeta_1, \cdots,\zeta_{k})\|_{[k+1;\mathbb R \times Z]}^2.
\]
\end{lem}

Let $m$ be a $[k; \mathbb R \times Z]$ multipliers. For $1 \le j \le
k$ we define the $j$-$support$ ${\rm supp}_j(m) \subset \mathbb{R}$
of $m$ to be the set
\[{\rm supp}_j(m) := \{\eta_j \in \mathbb{R} : \Ga_k(\mathbb{R} \times Z; \zeta_j = \eta_j) \cap {\rm supp}(m) \neq \emptyset\},\]
where $\Ga_k(\mathbb{R} \times Z; \zeta_j = \eta_j) = \{(\zeta_1,
\cdots, \zeta_k) \in \Ga_k(\mathbb R \times Z: \zeta_j = \eta_j)\}$.
And if $J$ is a non-empty subset of $\{1,\cdots,k\}$, we define the
set $supp_J(m) \subset \mathbb{R}^J$ by
\[{\rm supp}_J(m) := \prod_{j \in J}{\rm supp}_j(m).\]
\begin{lem}\label{schur}
Let $J_1, J_2$ be disjoint non-empty subsets of $\{1, \cdots, k\}$
and $A_1, A_2 > 0$. Suppose that $(m_a)_{a \in I}$ is a collection
of $[k; \mathbb R \times Z]$ multipliers such that
\[\#\{a \in I : \zeta \in {\rm supp}_{J_i}(m_a)\} \le A_i\]
for all $\zeta \in \mathbb{R}^{J_i}$ and $i = 1,2$. Then
\[\big\| \sum_{a \in I} m_a \big\|_{[k;\mathbb R \times Z]} \le (A_1A_2)^{\frac 12} \sup_{a \in I} \|m_a\|_{[k;\mathbb R \times Z]}.\]
In particular, if $m_a$ is non-negative and $A_1, A_2 \sim 1$, then
we have
\[\big\|\sum_{a \in I}m_a\big\|_{[k;\mathbb R \times Z]} \sim \sup_{a \in I}\|m_a\|_{[k;\mathbb R \times Z]}.\]
\end{lem}

 We set, for $j=1,2,3,$
 \[h_j = \pm |\xi_j|^\al, \,\,\, \zeta_j  = (\tau_j, \xi_j), \,\,\, \lambda_j = \tau_j -
h_j(\xi_j).\]   For the $X_Z^{s,b}$ space estimates, we need to
consider the $[3;\mathbb R \times Z]$-multiplier
$$
m(\zeta_1,\zeta_2, \zeta_3) = \frac{\widetilde
m(\xi_1,\xi_2,\xi_3)}{\prod_{j=1}^3 \langle \lambda_j \rangle^{b_j}}
$$
for a function $\widetilde m$ on $\mathbb R^3$ which will be
specified later. By averaging over unit time scale (Lemmas
\ref{compa} and \ref{trans}), one may restrict the multiplier to the
region $|\lambda_j| \ge 1$. And we define the function $h :
\Ga_3(\mathbb R \times Z) \rightarrow \mathbb{R}$ by setting
\[h(\xi_1,\xi_2,\xi_3) := h_1(\xi_1) + h_2(\xi_2) + h_3(\xi_3) =
-\lambda_1 - \lambda_2 - \lambda_3, \] which plays an important role
in what follows.

\newcommand{\sm}{\mathfrak{m}}
Let $N_j, L_j, H$ $(j = 1,2,3)$ be dyadic numbers. By dyadic
decomposition along the variables $\xi_j, \lambda_j$, as well as the
function $h(\xi_1, \xi_2, \xi_3)$, we have
\begin{align}\label{dysum}
\|m\|_{[3;\mathbb R \times Z]} \lesssim \Big\|\sum_{N_{max} \gtrsim 1} \sum_H \sum_{L_1,L_2,L_3 \gtrsim 1} \frac {\sm(N_1,N_2,N_3)}{L_1^{b_1}L_2^{b_2}L_3^{b_3}}X_{N_1,N_2,N_3;H;L_1,L_2,L_3}\Big\|_{[3;\mathbb R \times Z]},
\end{align}
where $X_{N_1,N_2,N_3;H;L_1,L_2,L_3}$ is the multiplier given by
\[X_{N_1,N_2,N_3;H;L_1,L_2,L_3}(\tau,\xi_1, \xi_2, \xi_3) := \chi_{\{|h(\xi_1, \xi_2, \xi_3)| \sim H\}} \prod_{j=1}^3 \chi_{\{|\xi_j|\sim N_j\}}\chi_{\{|\lambda_j| \sim L_j\}}\]
and \[\sm(N_1,N_2,N_3) := \sup_{|\xi_j| \sim N_j, \forall j = 1,2,3}|\widetilde m(\xi_1,\xi_2,\xi_3)|.\]
From the identities $\xi_1 + \xi_2 + \xi_3 = 0$ and $\lambda_1 + \lambda_2 + \lambda_3 + h(\xi_1, \xi_2, \xi_3) = 0$ on the support of the multiplier, we see that $X_{N_1,N_2,N_3;H;L_1,L_2,L_3}$ vanishes unless
\[N_{max} \sim N_{med}\;\;\mbox{and}\;\; L_{max} \sim \max(H, L_{med}).\]

Suppose for the moment that $N_1 \ge N_2 \ge N_3$. Then we have $N_1
\sim N_2 \gtrsim 1$. As $N_1$ ranges over the dyadic numbers, the
symbols in the summation in \eqref{dysum} are supported on
essentially disjoint regions of $\xi_1$ and $\xi_2$ spaces. This is
true for any permutation of $\{1,2,3\}$. Thus, by Lemma \ref{schur}
we have
\begin{align*}
\|m\|_{[3;\mathbb R \times Z]} \lesssim 
\sup_{N \gtrsim 1}\Big\| &\sum_{N_{max} \sim N_{med} \sim N} \sum_H \sum_{L_{max} \sim max(H,L_{med})}\\
&\frac {\sm(N_1,N_2,N_3)}{L_1^{b_1}L_2^{b_2}L_3^{b_3}}X_{N_1,N_2,N_3;H;L_1,L_2,L_3}\Big\|_{[3;\mathbb R \times Z]}.
\end{align*}
Hence, one is led to consider
\begin{align}\label{dyest}
\|X_{N_1,N_2,N_3;H;L_1,L_2,L_3}\|_{[3;\mathbb R \times Z]}
\end{align}
in the low modulation case $H \sim L_{max}$ and the high modulation
case $L_{max} \sim L_{med} \gg H.$ The following two lemmas give
estimates for \eqref{dyest} in each case.

\begin{lem}[(37) in \cite{tao}]\label{hmod}  If $L_{max} \sim L_{med} \gg H$, then
\begin{align*}
\eqref{dyest}
\lesssim L_{min}^\frac 12 \Big\|\mathcal{\chi}_{h(\xi) \sim H} \prod^3_{j=1} \mathcal{\chi}_{|\xi_j| \sim N_j}\Big\|_{[3;\mathbb{R}^{1+d}]} \lesssim L_{min}^\frac 12 |\{ \xi_2 \in Z : |\xi_2| \sim N_{min}\}|^\frac 12.
\end{align*}
\end{lem}

Let $|E|$ denote the Lebesgue measure or counting measure of any
measurable subset $E$ of $Z$.

\begin{lem}[Corollary 4.2 in \cite{tao}]\label{lmod}
Let $N_1, N_2, N_3 > 0, L_1 \ge L_2 \ge L_3$. Suppose that $H \sim
L_{max}$ and $\xi_1^0, \xi_2^0, \xi_3^0$ satisfy that
\[|\xi_j^0| \sim N_j \text { for } j = 1,2,3\;\;\text{and}\;\; |\xi_1^0 + \xi_2^0 + \xi_3^0| \ll N_{min}.\] Then we have
\[\eqref{dyest} \lesssim L_3^\frac 12 \big|\{\xi_2 \in Z : |\xi_2 - \xi_2^0| \ll N_{min}; h_2(\xi_2) + h_3(\xi - \xi_2) = \tau + \mathcal{O}(L_2)\}\big|^\frac 12 \]
for some $\tau \in \mathbb{R}$ and $\xi \in Z$ with $|\xi + \xi_1^0|
\ll N_{min}$. The same statement hold with the roles of the indices
1,2,3 permuted.
\end{lem}

\section{Bilinear Estimates}
In order to prove well-posedness for  \eqref{eq}, we show the
trilinear estimates (Proposition \ref{trilinear} below). For this
purpose, we first prove a bilinear estimate for $\|u\overline
v\|_{L^2(\mathbb R \times \widehat Z)}$, which automatically gives the estimate for $\|u v\|_{L^2(\mathbb R \times \widehat Z)}$. 
Since  the resonance function is $h(\xi_1, \xi_2, \xi_3) =
|\xi_1|^\al - |\xi_2|^\al + |\xi_3|^\al,$  we have\\
$|\xi_{max}|^{\alpha-1}|\xi_{min}| \lesssim |h(\xi)| \lesssim
|\xi_{max}|^\al.$

To begin with, we establish estimate for \eqref{dyest}. Here
$\langle \cdot \rangle_{Z}$ denotes $|\cdot|$ for non-periodic case
and $1 + |\cdot|$ for periodic case. So, $|\{\xi \in Z : a \le \xi
\le b\}| = O(\langle b-a \rangle_Z)$.
\begin{prop}\label{dyest2}
Let $H, N_1, N_2, N_3, L_1, L_2, L_3$ be dyadic and $h(\xi) =
|\xi_1|^\al - |\xi_2|^\al + |\xi_3|^\al$.  Then we have the
following.
\begin{itemize}
\item If $H \sim L_{max} \sim L_1$ and $N_{1} \sim N_{max}$,  
    \begin{itemize}
    \item[] $\eqref{dyest} \lesssim L_{min}^\frac 12 \langle\min(N_{min}^\frac 12, N_{max}^{\frac {1 - \al}2} L_{med}^\frac 12)\rangle_Z$.\\
    \end{itemize}
\item If $H \sim L_{max} \sim L_1$ and $N_{2} \sim N_{3} \gg N_{1}$,
    \begin{itemize}
    \item[] $\eqref{dyest} \lesssim L_{min}^\frac 12 \langle \min(N_{min}^\frac 12, N_{max}^{\frac {2 - \al}2} N_{min}^{-\frac 12} L_{med}^\frac 12)\rangle_Z$.\\
    \end{itemize}
\item If $H \sim L_{max} \sim L_2$ and $N_{max} \sim N_{min}$,  
    \begin{itemize}
    \item[] $\eqref{dyest} \lesssim L_{min}^\frac 12 \langle\min(N_{min}^\frac 12, N_{max}^{\frac {2 - \al}4} L_{med}^\frac 14)\rangle_Z$.\\
    \end{itemize}
\item If $H \sim L_{max} \sim L_2$ and $N_{max} \sim N_{med} \gg N_{min}$, 
    \begin{itemize}
    \item[] $\eqref{dyest} \lesssim L_{min}^\frac 12 \langle \min(N_{min}^\frac 12, N_{max}^{\frac {1 - \al}2} L_{med}^\frac 12)\rangle_Z$.\\
    \end{itemize}
\item If $H \ll L_{max} \sim L_{med}$, 
    \begin{itemize}
    \item[] $\eqref{dyest} \lesssim L_{min}^\frac 12 \langle N_{min}^\frac 12 \rangle_Z$.
    \end{itemize}
\end{itemize}
By symmetry, the same estimates also hold for the case $H \sim
L_{max} \sim L_3$.
\end{prop}
\begin{proof}
Lemma \ref{hmod} gives the high modulation case $H \ll L_{max} \sim
L_{med}$. So we need only to show the estimates in the first four
cases.

First we consider the  case $L_1 \sim L_{max}$ (the case $L_3 \sim
L_{max}$ follows by symmetry). Then by Lemma \ref{lmod}, we have
\begin{align}\label{length2}
\eqref{dyest} \lesssim L_3^\frac 12 \big|\{\xi_2 \in Z : |\xi_2 - \xi_2^0| \ll N_{min}; |\xi_2|^{\al} - |\xi - \xi_2|^{\al} = \tau + \mathcal{O}(L_2)\}\big|^\frac 12
\end{align}
for some $\tau \in \mathbb{R}$ and $\xi \in Z$ with $|\xi + \xi_1^0|
\ll N_{min}$. We observe that the derivative of $|\xi_2|^{\al} -
|\xi - \xi_2|^{\al}$ is equal to $\al(|\xi_2|^{\al-2}\xi_2 - |\xi_2
- \xi|^{\al -2}(\xi_2 - \xi))$.

If $N_1 \sim N_{\max}$, then $0 < |\xi_2| < C|\xi|$ for some
constant $C > 1$. This means $\xi_2$ is equal to $c\xi$ for some $0
< |c| < C$ and thus  $\big||\xi_2|^{\al-2}\xi_2 - |\xi_2 - \xi|^{\al
-2}(\xi_2 - \xi)\big| = \big|(|c|^{\al - 2}c - |c-1|^{\al -
2}(c-1))|\xi|^{\al -2}\xi\big|$, which is greater than or equal to
$((C+1)^{\al-1} - C^{\al-1})|\xi|^{\al - 1}$. So, $\xi_2$ is contained
in an interval of length $\mathcal{O}({N_{max}^{1-\al}L_{med}})$.
Hence, by \eqref{length2} we get the desired estimate for the first
case.

If $N_2 \sim N_3 \gg N_1$, then
\begin{align*}
&\qquad (\alpha-1)^{-1}\big||\xi_2|^{\al - 2}\xi_2 - |\xi_2 - \xi|^{\al - 2}(\xi_2 - \xi)\big| \\
&= \int_{\xi_2 - \xi}^{\xi_2} |\widetilde{\xi}|^{\al - 2} d\widetilde{\xi} \ge \min(|\xi_2|^{\al - 2}|\xi|, |\xi_2 - \xi|^{\al - 2}|\xi|).
\end{align*}
So, $\xi_2$ variable is contained in interval of length
$\mathcal{O}({N_{max}^{2-\al}N_{min}^{-1}L_{med}})$. This and
\eqref{length2} give the estimate for the second case.

We now consider the case $L_2 \sim L_{max}$. If $N_1 \sim N_2 \sim
N_3$, we see that \[\frac {|\xi_2|^{\al-2}\xi_2 + |\xi_2 - \xi|^{\al
-2}(\xi_2 - \xi)}{|\frac \xi 2|^{\al -2}(\xi_2 - \frac \xi2)}
\gtrsim 1\] by the Taylor expansion. This means that $\xi_2$ is
contained in an interval of length $\mathcal{O}(N_{max}^{\frac {2 -
\al}2}L_{med}^\frac 12)$ by the mean value theorem and the estimate
for the third case follows from \eqref{length2}.

If $N_{max} \sim N_{med} \gg N_{min}$, then we have
$||\xi_2|^{\al-2}\xi_2 +|\xi_2 - \xi|^{\al -2}(\xi_2 - \xi)| \sim
|\xi_2 - \frac \xi2|^{\al - 1} \sim N_{max}^{\al - 1}$ and thus
$\eqref{length2}$ and the mean value theorem shows that $\xi_2$ is
contained in an interval of length
$\mathcal{O}(N_{max}^{1-\al}L_{med})$. Since $\xi_2$ is also
contained in an interval of length $\ll N_{min}$, Proposition
\ref{dyest2} follows from \eqref{length2}.
\end{proof}
We now show some  bilinear estimates for the periodic and non
periodic cases.
\begin{prop}\label{bilinear2} Let $s
\ge \frac{2-\al}4$ and  $0 < \ep  \ll 1$. Then,  for  $u \in
X_{\widehat Z}^{0, \frac 12 - \ep}$ and $v \in X_{\widehat Z}^{s,
\frac 12 + \ep}$, we have
\[\|uv\|_{L^2(\mathbb{R} \times \widehat Z)} = \|u\overline{v}\|_{L^2(\mathbb{R} \times \widehat Z)} \lesssim \|u\|_{X_{\widehat Z}^{0, \frac 12 - \ep}}\|v\|_{X_{\widehat Z}^{s, \frac 12 + \ep}}.\]
\end{prop}
\noindent For the periodic case the following  is  to be useful.
\begin{lem}\label{aux}
\[\|u\overline{v}\|_{L^2(\mathbb{R} \times \mathbb T)} \lesssim \|(u - \widehat u(0))(\overline{v}-\widehat{\overline{v}}(0))\|_{L^2(\mathbb{R} \times \mathbb T)} + \|u\|_{X_{\mathbb T}^{0, \frac 12 - \ep}}\|\overline{v}\|_{X_{\mathbb T}^{0, \frac 12 + \ep}}.\]
\end{lem}
\begin{proof}[Proof of Lemma \ref{aux}] We observe
\begin{align*}
\|u\overline{v}\|_{L^2(\mathbb{R} \times \mathbb T)} &\le \|(u - \widehat u(0))(\overline{v}-\widehat{\overline{v}}(0))\|_{L^2(\mathbb{R} \times \mathbb T)} + \|u\widehat{\overline{v}}(0)\|_{L^2(\mathbb{R} \times \mathbb T)}\\
 &\quad + \|\widehat u(0) \overline{v}\|_{L^2(\mathbb{R} \times \mathbb T)} + \|\widehat u(0)\widehat{\overline{v}}(0)\|_{L^2(\mathbb{R} \times \mathbb T)}\\
&\le \|(u - \widehat u(0))(\overline{v}-\widehat{\overline{v}}(0))\|_{L^2(\mathbb{R} \times \mathbb T)} + \|u\|_{L^2(\mathbb{R} \times \mathbb T)}\|\widehat{\overline{v}}(0)\|_{L_t^\infty L_x^\infty}\\
&\quad + \|\widehat u(0)\|_{L_t^2L_x^\infty}\|\overline{v} \|_{L_t^\infty L_x^2} + \|\widehat u(0)\|_{L_t^2L_x^\infty}\|\widehat{\overline{v}}(0) \|_{L_t^\infty L_x^2}.
\end{align*}
By Sobolev embedding $X_{\mathbb T}^{0, \frac12+\ep} \hookrightarrow
C(\mathbb R; L^2(\mathbb T))$ we have
\[\|\widehat{\overline{v}}(0)\|_{L_t^\infty L_x^\infty} \le
\sqrt{2\pi}\|\overline{v}\|_{L_t^\infty L_x^2} \lesssim
\|\overline{v}\|_{X_{\mathbb T}^{0, \frac12+\ep}}\] and $\|\widehat
u(0)\|_{L_t^2L_x^\infty} \le \sqrt{2\pi}\|u\|_{L_{t,x}^2} \lesssim
\|u\|_{X_{\mathbb T}^{0, \frac12-\ep}}$. This gives the desired
estimate.
\end{proof}
\begin{proof}[Proof of Proposition \ref{bilinear2}]
For the proof it suffices to show that
\[\|\frac {1}{\langle \xi_1 \rangle^s \langle \tau_1 - |\xi_1|^\al \rangle^{\frac 12 + \ep}\langle \tau_2 + |\xi_2|^\al \rangle^{\frac 12 - \ep}}\|_{[3;\mathbb{R}\times Z]} \lesssim 1.\]
The left hand side is bounded by the sum of
\begin{align}\label{hmod2}
\sum_{N_{max} \sim N_{med} \sim N} \sum_{L_1,L_2,L_3 \gtrsim 1}\sum_{H\sim L_{max}}
\frac {1}{\langle N_1 \rangle^sL_1^{\frac 12 + \ep}L_2^{\frac 12 - \ep}}\|X_{N_1,N_2,N_3;H;L_1,L_2,L_3}\|_{[3;\mathbb{R}\times Z]},
\end{align}
and
\begin{align}\label{lmod2}
\sum_{N_{max} \sim N_{med} \sim N} \sum_{L_{max} \sim L_{med}\gtrsim 1} \sum_{H \ll L_{max}}
\frac {1}{\langle N_1 \rangle^s L_1^{\frac 12 + \ep}L_2^{\frac 12 - \ep}}\|X_{N_1,N_2,N_3;H;L_1,L_2,L_3}\|_{[3;\mathbb{R}\times Z]}.
\end{align}
From the Lemma \ref{aux} we may assume that $\widehat u(0) =
\widehat v(0) = 0$ and thus we may also assume that $N_{\min} \ge 1$
when $Z = \mathbb Z$.

Using Proposition \ref{dyest2}, we have
\[\eqref{lmod2} \lesssim \sum_{N_{max} \sim N_{med} \sim N} \sum_{L_{max} \sim L_{med} \gtrsim N_{max}^\al} \frac {1}{\langle N_1 \rangle^s  L_1^{\frac 12 + \ep}L_2^{\frac 12 - \ep}} L_{min}^\frac 12\langle N_{min}^\frac 12\rangle_Z.\]
Since $L_1^{\frac 12 + \ep}L_2^{\frac 12 - \ep} \gtrsim L_{min}^{\frac 12 + \ep}L_{med}^{\frac 12 - \ep}$, we get
\begin{align*}
&\qquad \eqref{lmod2} \lesssim \sum_{N_{max} \sim N_{med} \sim N} \frac {1}{\langle N_1 \rangle^s N_{max}^{\frac \al2 - \ep\al}} N_{min}^\frac 12\\ 
&\lesssim \sum_{N_{min} \lesssim N} \frac {1}{\langle N_{min} \rangle^s N^{\frac \al2 - \ep\al}} N_{min}^\frac 12 \lesssim 1 + N^{\frac 12 - \frac \al2 + \ep\al} \lesssim 1.
\end{align*}


Now we turn to \eqref{hmod2}. Firstly we consider the case $L_1 =
L_{max}$ and $N_{\max} = N_1$ (the estimate for the case $L_3 =
L_{max}$ and $N_{\max} = N_3$ follow by symmetry). Proposition
\ref{dyest2} gives
\begin{align*}
\eqref{hmod2}
&\lesssim \sum_{N_{max} \sim N_{med} \sim N}  \sum_{L_1,L_2,L_3 \gtrsim 1}\sum_{H\sim L_1}
\frac {L_{min}^\frac 12 \langle\min(N_{min}^\frac 12, N_{max}^{\frac {1 - \al}2} L_{med}^\frac 12)\rangle_Z}{\langle N_1 \rangle^s  L_1^{\frac 12 + \ep}L_2^{\frac 12 - \ep}}\\
&\lesssim \sum_{N_{max} \sim N_{med} \sim N}\sum_{L_1,L_2,L_3 \gtrsim 1}
\frac {L_{min}^\frac 12\langle \min(N_{min}^\frac 12, N_{max}^{\frac {1 - \al}2} L_{med}^\frac 12)\rangle_Z}{\langle N_{\min} \rangle^s L_1^{\frac 12 + \ep}L_{med}^{\frac 12 - \ep}}.
\end{align*}
Here $H$-sum is bounded by an absolute constant. By summing in
$L_{min}$ and then $L_1$, we get
\begin{align*}
\eqref{hmod2}
&\lesssim \sum_{N_{max} \sim N_{med} \sim N} \sum_{L_{max}\ge L_{med} \ge 1}
\frac {\langle\min(N_{min}^\frac 12, N_{max}^{\frac {1 - \al}2} L_{med}^\frac 12)\rangle_Z L_{med}^\ep}{\langle N_{\min} \rangle^s L_{max}^{\frac 12 + \ep}}.
\end{align*}
If $Z = \mathbb R$, then we separate $N_{min}$ sum as follows:
\begin{align*}
&\qquad \eqref{hmod2}\\ 
&\lesssim \left(\sum_{0 < N_{min} < N^{1-\al} } + \sum_{N^{1-\al} \le N_{min} \lesssim N}\right) \sum_{L_{med} \ge 1}\frac {\min(N_{min}^\frac 12 L_{med}^{-\frac 12}, N^{\frac {1 - \al}2})}{\langle N_{\min} \rangle^s }\\
&\lesssim \sum_{N_{min} < N^{1-\al} }\sum_{L_{med} \ge 1}\frac {N_{min}^\frac 12 L_{med}^{-\frac 12}}{\langle N_{\min} \rangle^s } + \sum_{N_{min} = N^{1-\al}}^{N}\sum_{L_{med} \ge 1}\frac {\min(N_{min}^\frac 12 L_{med}^{-\frac 12}, N^{\frac {1 - \al}2})}{\langle N_{\min} \rangle^s }\\
&\lesssim N^{\frac{1-\al}2} + \sum_{N_{min}=N^{1-\al}}^{N}\left(\sum_{1 \le L_{med} < N_{min}N^{\al-1}}N^{\frac {1 - \al}2} + \sum_{L_{med} \ge N_{min}N^{\al-1}}N_{min}^\frac 12 L_{med}^{-\frac 12}\right)\\
&\lesssim N^{\frac{1-\al}2}
+ N^{(\al -1)(-\frac 12 + \ep) + \ep} \lesssim 1.
\end{align*}
If $Z = \mathbb Z$, then we have
\begin{align*}
&\qquad \eqref{hmod2}\\
&\lesssim \sum_{N_{min}=1}^{N} \left( \sum_{N^{\al - 1}N_{min} \le L_{med} \le L_{max}} + \sum_{L_{med} \le N^{\al - 1}N_{min}} \right) \frac {(1 + \min(N_{min}^\frac 12, N^{\frac {1 - \al}2}L_{med}^{\frac 12})) L_{med}^\ep}{\langle N_{\min} \rangle^s L_1^{\frac 12 + \ep}}\\
&\lesssim \sum_{N_{min}=1}^{N} \left( \sum_{N^{\al - 1}N_{min} \le L_{med} \le L_{max}} \frac {N_{min}^\frac 12 L_{med}^{\ep}}{N_{\min}^s L_{max}^{\frac 12 + \ep}} + \sum_{L_{med} \le N^{\al - 1}N_{min}}  \frac {(1 + N^{\frac {1 - \al}2} L_{med}^{\frac 12}) L_{med}^{\ep}}{N_{\min}^s L_{max}^{\frac 12 + \ep}}\right)\\
&\lesssim \sum_{N_{min}=1}^{N} \sum_{N^{\al - 1}N_{min} \le L_{max}}\frac {N_{min}^{\frac 12}}{N_{\min}^s L_{max}^{\frac 12}}
+ \sum_{N_{min}=1}^{N} \sum_{L_{med} \le N^{\al - 1}N_{min}} \frac {(1 + N^{\frac {1 - \al}2} L_{med}^{\frac 12}) L_{med}^{\ep}}{N_{\min}^s L_{max}^{\frac 12 + \ep}}\\
&\lesssim \sum_{N_{min}=1}^{N} \frac {N_{min}^{\frac 12}}{N_{\min}^s (N^{\al - 1}N_{min})^{\frac 12}}
+ \sum_{N_{min}=1}^{N} \sum_{L_{med} \le N^{\al - 1}N_{min}} \frac {(1 + N^{\frac {1 - \al}2} L_{med}^{\frac 12}) L_{med}^{\ep}}{N_{\min}^s L_{max}^{\frac 12 + \ep}}\\
&\lesssim N^{\frac {1-\al}{2}} + \sum_{N_{min}=1}^{N} \sum_{L_{med} \le N^{\al - 1}N_{min}} \frac {(1 + N^{\frac {1 - \al}2} L_{med}^{\frac 12}) L_{med}^{\ep}}{N_{\min}^s L_{max}^{\frac 12 + \ep}}\\
&\lesssim N^{\frac {1-\al}{2}} + \sum_{N_{min}=1}^{N}\left(\sum_{1 \le L_{med} < N^{\al-1}} \frac {L_{med}^{\ep}}{N_{\min}^s L_{max}^{\frac 12 + \ep}} + \sum_{N^{\al-1} \le L_{med} \le N^{\al-1}N_{min}}\frac {N^{\frac {1 - \al}2} L_{med}^{\frac 12} L_{med}^{\ep}}{N_{\min}^s L_{max}^{\frac 12 + \ep}}\right)\\
&\lesssim N^{\frac {1-\al}{2}} +1
+ \sum_{N_{min}=1}^{N}\sum_{N^{\al-1} \le L_{med} \le N^{\al-1}N_{min}} N_{min}^{-s}N^{\frac {1 - \al}2}\\
&\lesssim 1+ N^{\frac{1-\al}2}\log{N} \lesssim 1.
\end{align*}

Secondly, we deal with the case $L_2 = L_{max}$ and $N_{max} \sim
N_{min}$. Using Proposition \ref{dyest2}, we have
\[\eqref{lmod2} \lesssim \sum_{N_{max} \sim N_{min} \sim N} \sum_{L_{max} \geq L_{med} \geq L_{min} \gtrsim 1} \frac {1}{\langle N_1 \rangle^s  L_1^{\frac 12 + \ep}L_2^{\frac 12 - \ep}} L_{min}^\frac 12 \langle\min(N_{min}^\frac 12, N_{max}^{\frac {2 - \al}4} L_{med}^\frac 14) \rangle_Z.\]
Since $s \ge \frac{2-\al}4$, we have
\[\eqref{lmod2} \lesssim \sum_{N_{max} \sim N_{min} \sim N} \sum_{L_{med} \ge 1}\frac {1}{\langle N \rangle^s L_{med}^{\frac12-\ep}} \langle N^{\frac {2 - \al}4}L_{med}^\frac14 \rangle_Z \lesssim 1.\]

We now handle the remaining three cases: $L_1 = L_{max}$ and $N_2
\sim N_3 \gg N_1$; $L_2 = L_{max}$ and $N_3 \sim N_1 \gg N_2$; $L_3
= L_{max}$ and $N_1 \sim N_2 \gg N_3$.

\subsection*{Case $L_1 = L_{max}$ and $N_2 \sim N_3 \gg N_1$} Since
$N_2 \sim N_3 \gg N_1$ and $\xi_1 + \xi_2 + \xi_3 = 0$, one can
observe that $H \sim |h(\xi_1, \xi_2, \xi_3)| \sim
|\xi_1||\xi_{\max}|^{\al-1} \sim N_{min}N^{\al-1}$. Thus we have $1
\lesssim L_1 \sim H \sim N^{\al-1}N_{min}$, which means that
$N_{min} \gtrsim N^{1-\al}$. Using Proposition \ref{dyest2} and
performing  $L_{min}$ and $L_1$ summation, we have
\begin{align*}
&\qquad \eqref{hmod2}\\
&\lesssim \sum_{\substack{N_{max} \sim N_{med} \sim N\\N_{min} \ge N^{1-\al} }} \sum_{L_{max} \sim N^{\al-1}N_{min}} \sum_{L_{med} \ge L_{min} \gtrsim 1}
\frac {L_{min}^\frac 12 \langle\min(N_{min}^\frac 12, N_{max}^{\frac {2 - \al}2} N_{min}^{-\frac 12} L_{med}^\frac 12)\rangle_Z}{\langle N_1 \rangle^s  L_1^{\frac 12 + \ep}L_2^{\frac 12 - \ep}}\\
&\lesssim  \sum_{\substack{N_{max} \sim N_{med} \sim N\\N_{min} \ge N^{1-\al} }}\sum_{L_{max} \sim N^{\al-1}N_{min}} \sum_{L_{med} \ge L_{min} \gtrsim 1}
\frac {L_{min}^\ep \langle\min(N_{min}^\frac 12, N_{max}^{\frac {2 - \al}2} N_{min}^{-\frac 12} L_{med}^\frac 12)\rangle_Z}{\langle N_{min} \rangle^s L_{max}^{\frac 12 + \ep}}\\
&\lesssim  \sum_{N^{1-\al} \le N_{min} \lesssim N} \sum_{N^{\al-1}N_{min} \ge L_{med} \ge 1}
\frac {L_{med}^\ep\langle\min(N_{min}^\frac 12, N^{\frac {2 - \al}2} N_{min}^{-\frac 12}L_{med}^{\frac12})\rangle_Z}{\langle N_{min} \rangle^s (N^{\al-1}N_{min})^{\frac12+\ep}}.
\end{align*}

When $Z = \mathbb R$, by separating $N_{min}$ sum into the cases
$N_{min} < N^{\frac{1-\al}2}$ and $N_{min} \ge N^{\frac{1-\al}2}$,
we have
\begin{align*}
\eqref{hmod2}
&\lesssim  \sum_{N^{1-\al} \le N_{min} < N^{\frac{1-\al}2}}  \sum_{N^{\al-1}N_{min} \ge L_{med} \ge 1}\frac {N_{min}^\frac 12 L_{med}^{\ep}}{\langle N_{min} \rangle^s(N^{\al-1}N_{min})^{\frac12+\ep}}\\
&\qquad + \sum_{N^{\frac{1-\al}2} \le N_{min} \lesssim N}  \sum_{N^{\al-1}N_{min} \ge  L_{med} \ge 1}
\frac {\min(N_{min}^\frac 12 L_{med}^{\ep}, N^{\frac {2 - \al}2} N_{min}^{-\frac 12}L_{med}^{\frac12+\ep})}{\langle N_{min} \rangle^s (N^{\al-1}N_{min})^{\frac12+\ep}}\\
&\lesssim N^{\frac{1-\al}2+\ep} + \sum_{N^{\frac{1-\al}2} \le N_{min} \lesssim N} \sum_{N^{\al-1}N_{min} \ge  L_{med} \ge 1}\frac {N_{min}^\frac 12 L_{med}^{\ep}}{\langle N_{min} \rangle^s (N^{\al-1}N_{min})^{\frac12+\ep}}\\
&\lesssim N^{\frac{1-\al}2+\ep} \lesssim 1.
\end{align*}

Otherwise ($Z = \mathbb Z$), since $N_{min} \ge 1$, we have
\begin{align*}
\eqref{hmod2}
&\lesssim \sum_{1 \le N_{min} \lesssim N} \sum_{N^{\al-1}N_{min} \ge L_{med} \ge 1}
\frac {L_{med}^\ep (1 + \min(N_{min}^\frac 12, N^{\frac {2 - \al}2} N_{min}^{-\frac 12}L_{med}^{\frac12}))}{\langle N_{min} \rangle^s (N^{\al-1}N_{min})^{\frac12+\ep}}\\
&\lesssim  1 + \sum_{1 \le N_{min} \lesssim N}  \sum_{N^{\al-1}N_{min} \ge  L_{med} \ge 1}
\frac {\min(N_{min}^\frac 12 L_{med}^{\ep}, N^{\frac {2 - \al}2} N_{min}^{-\frac 12}L_{med}^{\frac12+\ep})}{\langle N_{min} \rangle^s (N^{\al-1}N_{min})^{\frac12+\ep}}\\
&\lesssim 1 + \sum_{1 \le N_{min} \lesssim N} \sum_{N^{\al-1}N_{min} \ge  L_{med} \ge 1}\frac {N_{min}^\frac 12 L_{med}^{\ep}}{\langle N_{min} \rangle^s (N^{\al-1}N_{min})^{\frac12+\ep}}\\
&\lesssim 1 + N^{\frac{1-\al}2}\log N \lesssim 1.
\end{align*}


\subsection*{Case $L_2 = L_{max}$ and $N_3 \sim N_1 \gg N_2$}  In
this case we have $L_2 \sim H \sim N^{\al}$. From Proposition
\ref{dyest2}, summation in $L_{min}$ and the assumption $N_{min} \ge
1$ for $Z = \mathbb Z$, we have
\begin{align*}
\eqref{hmod2}
&\lesssim \sum_{\substack{N_{max} \sim N_{med} \sim N\\N_{min} \ge N^{1-\al} }} \sum_{L_{max} \sim N^\al} \sum_{ L_{med} \ge L_{min} \gtrsim 1}
\frac {L_{min}^\frac 12 \langle\min(N_{min}^\frac 12, N_{max}^{\frac {1 - \al}2} L_{med}^\frac 12)\rangle_Z}{\langle N_1 \rangle^s L_1^{\frac 12 + \ep}L_{max}^{\frac 12 - \ep}}\\
&\lesssim  \sum_{\substack{N_{max} \sim N_{med} \sim N\\N_{min} \ge N^{1-\al} }} \sum_{1 \lesssim L_{med} \le N^\al}
\frac {\langle\min(N_{min}^\frac 12, N_{max}^{\frac {1 - \al}2}  L_{med}^{\frac 12})\rangle_Z}{N^{s+\al(\frac 12 - \ep)}}\\
&\lesssim  \sum_{N^{1-\al} \le N_{min} \lesssim  N}
\frac {\langle N_{min}^\frac 12\rangle_Z\log N}{N^{s+\al(\frac 12 - \ep)}} \lesssim 1.
\end{align*}

\subsection*{Case $L_3 = L_{max}$ and $N_1 \sim N_2 \gg N_3$} In
this case $L_3 \sim H \sim N^{\al-1}N_{min}$. By Proposition
\ref{dyest2} and summation in $L_{min}$, we have
\begin{align*}
&\qquad \eqref{hmod2}\\
&\lesssim \sum_{\substack{N_{max} \sim N_{med} \sim N\\N_{min} \ge N^{1-\al} }} \sum_{L_{max} \sim N^{\al-1}N_{min}} \sum_{L_{med} \ge L_{min} \gtrsim 1}
\frac {L_{min}^\frac 12\langle \min(N_{min}^\frac 12, N_{max}^{\frac {2 - \al}2} N_{min}^{-\frac 12} L_{med}^\frac 12)\rangle_Z}{\langle N_1 \rangle^s  L_1^{\frac 12 + \ep}L_2^{\frac 12 - \ep}}\\
&\lesssim  \sum_{\substack{N_{max} \sim N_{med} \sim N\\N_{min} \ge N^{1-\al} }} \sum_{L_{max} \sim N^{\al-1}N_{min}} \sum_{L_{med} \ge L_{min} \gtrsim 1}
\frac {L_{min}^\frac 12 \langle \min(N_{min}^\frac 12, N_{max}^{\frac {2 - \al}2} N_{min}^{-\frac 12} L_{med}^\frac 12)\rangle_Z}{ N^{s} L_{min}^{\frac 12 + \ep}L_{med}^{\frac 12 - \ep}}\\
&\lesssim  \sum_{\substack{N_{max} \sim N_{med} \sim N\\N_{min} \ge N^{1-\al} }} \sum_{L_{max} \sim N^{\al-1}N_{min}} \sum_{L_{med} \gtrsim 1}
\frac {\langle\min(N_{min}^\frac 12 , N^{\frac {2 - \al}2} N_{min}^{-\frac 12} L_{med}^{\frac12})\rangle_Z}{N^{s}L_{med}^{\frac 12 - \ep}}.
\end{align*}
Since $N^{\al-1}N_{min} \sim L_{max} \gtrsim 1$ implies $N_{min}
\gtrsim N^{1 - \al}$, by breaking $N_{min}$-sum into two parts, we
have:
\begin{align*}
&\qquad \eqref{hmod2}\\
&\lesssim  \sum_{N^{1 - \al} \le N_{min} \le N^{\frac {2 - \al}2}}\frac {\langle N_{min}^\frac 12\rangle_Z}{N^{s}}\\
&\qquad + \sum_{N^{\frac {2 - \al}2} < N_{min} \lesssim N} \sum_{L_{max} \sim N^{\al-1}N_{min}} \sum_{L_{med} \gtrsim 1}
\frac {\langle\min(N_{min}^\frac 12 , N^{\frac {2 - \al}2} N_{min}^{-\frac 12} L_{med}^{\frac12})\rangle_Z}{N^{s}L_{med}^{\frac 12 - \ep}}
\\
&\lesssim N^{\frac{2-\al}4 - s} + \sum_{N^{\frac {2 - \al}2} < N_{min} \lesssim N} \sum_{L_{max} \sim N^{\al-1}N_{min}} \sum_{L_{med} \gtrsim 1}
\frac {\langle\min(N_{min}^\frac 12 , N^{\frac {2 - \al}2} N_{min}^{-\frac 12} L_{med}^{\frac12})\rangle_Z}{N^{s}L_{med}^{\frac 12 - \ep}}.
\end{align*}
For the second inequality we use $\sum_{N^{1 - \al} \le N_{min} \le
N^{\frac {2 - \al}2}} ({1+ N_{min}^\frac 12}){N^{-s}} \lesssim
N^{\frac{2-\al}{4}-s} + N^{-s}\log N \lesssim
N^{\frac{2-\al}{4}-s}.$ Now by dividing $L_{med}$-sum into\\ $\sum_{1 \leq L_{med} \leq N^{\al - 2}N_{min}^2} + \sum_{N^{\al - 2}N_{min}^2 < L_{med}}$, we  get
\begin{align*}
\eqref{hmod2}
&\lesssim N^{\frac{2-\al}4 - s} + \sum_{N^{\frac {2 - \al}2} < N_{min} \lesssim N} \sum_{L_{max} \sim N^{\al-1}N_{min}} \sum_{L_{med} \gtrsim 1}
\frac {\min(N_{min}^\frac 12 , N^{\frac {2 - \al}2} N_{min}^{-\frac 12} L_{med}^{\frac12})}{N^{s}L_{med}^{\frac 12 - \ep}}\\
&\lesssim N^{\frac{2-\al}4 - s} + \sum_{N^{\frac {2 - \al}2} <
N_{min} \lesssim N} \frac {N_{min}^\frac
12}{N^s(N^{\al-2}N_{min}^2)^{\frac12-\ep}}\lesssim  N^{\frac{2-\al}4
- s}.
\end{align*}
Since $s \ge \frac{2-\al}{4}$, we get the desired result.
\end{proof}

\section{Proof of Theorem \ref{main1}}
For the proof Theorem \ref{main1}, we need the trilinear estimate

\begin{equation}\label{tri}
\|u_1\overline{u_2}u_3\|_{X_{\mathbb R}^{s, -\frac 12 + \ep}} \lesssim
\prod_{j=1}^3 \|u_j\|_{X_{\mathbb R}^{s, \frac 12 + \ep}}(\mathbb{R} \times
\mathbb{R}).
\end{equation}

\subsubsection*{Failure of \eqref{tri} for $s< \frac{2-\alpha}{4}$}It is easy to see that the trilinear estimate fails
when $s < \frac{2-\alpha}{4} $. The counter-example is a resonant
high-high-high to high interaction. For $N \gg 1$, let
\begin{align*}
 \widetilde{u_1},\widetilde{u_3} &= \chi_{A_N}, \qquad A_N =\{(\xi,\tau): N\le \xi \le  N+N^{\frac{2-\alpha}{2}} , \quad |\tau -|\xi|^\alpha | \le 1 \}, \\
\widetilde{\overline{u_2}} &= \chi_{A_N}, \qquad A_N =\{(\xi,\tau): -N\le \xi \le  -N+N^{\frac{2-\alpha}{2}} , \quad |\tau +|\xi|^\alpha | \le 1 \}.
\end{align*}
Here, the number $ N^{\frac{2-\alpha}{2}} $ is chosen so that the
parallelogram $A_N$ to be fit in a width 1 strip of $\tau =
|\xi|^\alpha $. Then,  it follows that
\begin{align*}
\|\wt{u_1} * \wt{\overline u_2} * \wt{u_3}  \|_{X^{s,b-1}} &\sim
N^{\frac{2-\alpha}{2}} N^{\frac{2-\alpha}{2}} N^s
N^{\frac{2-\alpha}{4}}, \,\, \mbox{and}\;\;\|u_j\|_{X^{s,b}} \sim
N^sN^{\frac{2-\alpha}{4}}.
\end{align*}
This and letting $N\to \infty$ give the necessary condition $ s \ge
\frac{2-\alpha}{4} $  for \eqref{tri} .

\begin{prop}\label{trilinear} Let $s \ge \frac{2-\al}{4}$ and $0 < \ep \ll
1$. For any $u_1,u_2,$ and $u_3\in X_{\widehat Z}^{s, \frac12+\ep}$,
we have
\[\|u_1\overline{u_2}u_3\|_{X_{\widehat Z}^{s, -\frac 12 + \ep}} \lesssim \prod_{j=1}^3 \|u_j\|_{X_{\widehat Z}^{s, \frac 12 + \ep}}.\]
\end{prop}

\begin{proof}
By duality and Plancherel's theorem it suffices to show that
\[\Big\|\frac{\langle \xi_4 \rangle^s}{\langle \xi_1 \rangle^s \langle \xi_2 \rangle^s \langle \xi_3 \rangle^s \langle \tau_1 - |\xi_1|^\al\rangle^{\frac12 + \ep} \langle \tau_2 + |\xi_2|^\al\rangle^{\frac12 + \ep}\langle \tau_3 - |\xi_3|^\al\rangle^{\frac12 + \ep} \langle \tau_4 + |\xi_4|^\al\rangle^{\frac12 - \ep}}\Big\|_{[4;\mathbb{R}\times Z]} \lesssim 1.\]
Since $\langle \tau_2 + |\xi_2|^\al \rangle^{\frac 12 + \ep} \gtrsim  \langle \tau_2 + |\xi_2|^\al \rangle^{\frac 12 - \ep}$, the desired estimate follows from Lemma \ref{compo} and bilinear estimates Propositions \ref{bilinear2}.
\end{proof}

\begin{proof}[Proof of Theorem \ref{main1}]
We define a nonlinear functional $\mathcal{N}$ by
\[\mathcal{N}(u) = \psi(t)U(t)\phi - i\gamma\psi(t/T)\int^t_0 U(t-t')|u|^2u(t') dt',\]
where $\psi$ is a fixed smooth cut-off function such that $\psi(t) =
1$ if $|t| < 1$ and $\psi(t) = 0$ if $|t|>2$, and $0 < T \le 1$ is
fixed. For $s,b \in \mathbb{R}$ we define the norm $X_{\widehat
Z}^{s,b}$  for on the time interval $J_T = [0,T]$ by
\[\|u\|_{X_{\widehat Z}^{s,b}(J_T)} := \inf\big\{\|v\|_{X_{\widehat Z}^{s,b}}: v|_{J_T} = u\big\}.\]
Then we recall  the well-known properties of $X_{\widehat Z}^{s,b}$
:
\begin{align}\label{li-est}
\|\psi(t)U(t)\phi\|_{X_{\widehat Z}^{s,b}} \lesssim \|\phi\|_{H^s},
\,\, b\in \mathbb R,
\end{align}
 and, for $-\frac 12 < b' \le 0, 0 \le b \le b' + 1,$
\begin{align}\label{nonli-est}
\|\int^t_0 U(t-t')F(t',x)dt'\|_{X_{\widehat Z}^{s,b}(J_T)} \lesssim
T^{1+b'-b}\|F\|_{X_{\widehat Z}^{s,b'}(J_T)}.
\end{align}

Define a compete metric space $B_{T,\rho}$ by
\[B_{T,\,\rho} = \big\{u \in X_{\widehat Z}^{s, \frac 12 + \ep}(J_T) : \|u\|_{X_{\widehat Z}^{s, \frac 12 + \ep}}(J_T) \le \rho\big\}\]
with the metric $d(u,v) = \|u-v\|_{X_{\widehat Z}^{s, \frac 12 +
\ep}(J_T)}$. From \eqref{li-est} and \eqref{nonli-est} with $b=\frac
12 + \ep, b' = - \frac 12 + \ep', \ep < \ep'$ it follows that, for
any $u \in B_{T,\rho}$,
\[\|\mathcal{N}(u)\|_{X_{\widehat Z}^{s, \frac 12 + \ep}(J_T)} \lesssim \|\phi\|_{H^s} + T^{\ep' - \ep}\||u|^2u\|_{X_{\widehat Z}^{s, -\frac 12 - \ep'}(J_T)}.\]
If $\ep'$ is sufficiently small, from Proposition \ref{trilinear} we
see
\[\|\mathcal{N}(u)\|_{X_{\widehat Z}^{s, \frac 12 + \ep}(J_T)} \lesssim \|\phi\|_{H^s} + T^{\ep' - \ep}\|u\|_{X_{\widehat Z}^{s, \frac 12 + \ep}(J_T)} \lesssim \|\phi\|_{L^2_x} + T^{\ep' - \ep}\rho^3.\]
Choosing $\rho$ and $T$ small enough  so that $\rho \ge
2C\|\phi\|_{H^s}$ and $CT^{\ep' - \ep}\rho^3 \le \rho/2$ for some
constant $C$,
 we see that the functional $\mathcal{N}$ is a map from $B_{T,\rho}$ to itself.
 Similarly one can show that $N(u)$ is a contraction.
 Therefore there is a unique $u \in X_{\widehat Z}^{s, \frac 12 + \ep}(J_T)$
 satisfying \eqref{integral}.
\end{proof}

\section{Ill-posedness}
In this section, we prove that the  equation \eqref{eq} in the
non-periodic case is ill-posed for $\frac {2 - 3\al}{4(\al + 1)} < s
< \frac {2 - \al}4$. For convenience we assume that $\gamma = 1$.
Our strategy is to approximate the solution by the solutions of
\eqref{eq2} which is ill-posed in $H^s, s < 0$ (see \cite{cct} for
the non-periodic case and \cite{cct2, mol} for the periodic one).
For this purpose we recall ill-posedness result for the
Schr\"odinger equation
\begin{align}\label{eq2}
\left\{\begin{array}{l}
i\partial_tv -\Delta v = |v|^2v,\\
v(0,\cdot) = \phi \in H^s.
\end{array}
\right.
\end{align}

\begin{thm}\label{ill-NLS}
Let $s<0$. The solution map of the initial value problem of the
cubic NLS \eqref{eq2} fails to be uniformly continuous. More
precisely, for $0 < \delta \ll \ep \ll 1$ and $T > 0$ arbitrary,
there are two solutions $v_1, v_2$ to \eqref{eq2} with initial data
$\phi_1, \phi_2$, respectively, satisfying \eqref{ill-1},
\eqref{ill-2} and \eqref{ill-3}. Moreover we can find solutions to
satisfy
\begin{align}\label{ill-sm}
\sup_{0 \leq t \leq \infty}\|v_j(t)\|_{H^5} \lesssim \ep,
\end{align}
for $j=1,2$.
\end{thm}

Let $N \gg 1$ be a large parameter to be chosen later. Let $v(s,y)$
be a solution of the cubic NLS equation \eqref{eq2} and
\begin{equation} \label{changevar}(s,y) := \Big(t, \frac {x + \al N^{\al - 1}t}{(\frac
{\al(\al - 1)}2 N^{\al - 2})^{\frac 12}}\Big).\end{equation}  We
shall construct approximate solutions which is  given by
\begin{align*}\label{appsol}
V(t,x) := e^{iNx} e^{iN^\al t} v(s,y).
\end{align*}
It is easy to see that
\begin{align*}
&\qquad (i\partial_t + (-\Delta)^{\frac \al 2})V\\
&= e^{iNx}e^{iN^\al t} \Big(- N^\al v(s,y) + i \partial_s v(s,y) + i\frac {\al N^{\al - 1}}{(\frac {\al(\al - 1)}2 N^{\al - 2})^{\frac 12}}
\partial_y v(s,y)\Big) \\
& + e^{iNx}e^{iN^\al t}\Big(N^\al v(s,y) - i \frac {\al
N^{\al - 1}}{(\frac {\al(\al - 1)}2 N^{\al - 2})^{\frac 12}}
\partial_y v(s,y) -  \partial_{yy} v(s,y) + R(-i\partial_y)
v(s,y)\Big),
\end{align*}
where \begin{equation}\label{remain} R(\xi) = \Big|\frac
{\xi}{\big(\frac {\al(\al - 1)}2 N^{\al - 2}\big)^{\frac 12}} +
N\Big|^\al - N^\al - \frac {\al N^{\al - 1}}{\big(\frac {\al(\al -
1)}2 N^{\al - 2}\big)^{\frac 12}}\xi - \xi^2.\end{equation}

Since $v(s,y)$ is a solution of \eqref{eq2}, we have
\[iV_t + (-\Delta)^\frac {\al}2V - |V|^2V = E,\]
where $E=e^{iNx} e^{iN^\al t} R(-i\partial_y) v(s,y).$ We need to
bound the error. First we show the following perturbation result
relying on the local well-posedness.
\begin{lem}\label{pertub}
Let $u$ be a smooth solution to \eqref{eq} and $V$ be a smooth
solution to the equation
\[iV_t + (-\Delta)^\frac {\al}2V - |V|^2V = \mathcal E\]
for some error function $\mathcal E$. Let $e$ be the solution to the
inhomogeneous problem $ie_t + (-\Delta)^\frac {\al}2e = \mathcal
E,\; e(0) = 0$ and let $\eta(t)$ be a compactly supported smooth
time cut-off function such that $\eta=1$ on $J = [0,1]$. Suppose
that \\ $\|u(0)\|_{H^{\frac {2-\al}4}}, \;\;\|V(0)\|_{H^{\frac
{2-\al}4}},\;\; \|\eta(t)e\|_{X_{\mathbb R}^{\frac
{2-\al}4,\frac12+}} \lesssim \ep.$  Then, if $\epsilon$ is
sufficiently small,  we have
\[\|(u-V)\|_{X_{\mathbb R}^{\frac {2-\al}4,\frac12+}(J)} \lesssim \|u(0) - V(0)\|_{H^{\frac {2-\al}4}} + \|\eta(t)e\|_{X_{\mathbb R}^{\frac {2-\al}4,\frac12+}}.\]
In particular, we have
\[\sup_{0 \leq t \leq 1}\|u(t)-V(t)\|_{H^{\frac {2-\al}4}} \lesssim \|u(0) - V(0)\|_{H^{\frac {2-\al}4}} + \|\eta(t)e\|_{X_{\mathbb R}^{\frac {2-\al}4,\frac12+}}.\]
\end{lem}
\begin{proof}
Writing the equation for $V$ in integral form, we have
\[V(t) = U(t)V(0) + e(t) - i\int_0^t U(t - t')(|V|^2V)(t')dt'.\]
By taking $X_{\mathbb R}^{\frac {2-\al}4, \frac12+}(J)$ norm on both
sides and applying \eqref{nonli-est}, we get
\begin{align*}
\|V\|_{X_{\mathbb R}^{\frac {2-\al}4,\frac12+}(J)} &\lesssim \|V(0)\|_{H^{\frac {2 - \al}4}} + \|\eta(t)e\|_{X_{\mathbb R}^{\frac{2-\al}4,\frac12+}} + \||V|^2V\|_{X_{\mathbb R}^{\frac {2-\al}4,-\frac12+}(J)}\\
 &\lesssim \|V(0)\|_{H^{\frac {2 - \al}4}} + \|\eta(t)e\|_{X_{\mathbb R}^{\frac {2-\al}4,\frac12+}} + \|V\|_{X_{\mathbb R}^{\frac{2-\al}4,\frac12+}(J)}^3.\end{align*}
By continuity argument with sufficiently small $\ep$, we obtain
$\|V\|_{X_{\mathbb R}^{\frac {2-\al}4,b}(J)} \lesssim \ep.$

Let $w := u - V$. Then $w$ satisfies the equation
\[iw_t + (-\Delta)^{\al/2}w = |w|^2w + 2|w|^2v + 2w|v|^2 + w^2\bar{v} + \bar{w}v^2 - E,\;\; w(0) = u(0) - V(0),\]
which is written in integral form as
\[w(t) = U(t)w(0) - e(t) - i\int_0^t U(t - t')(|w|^2w + 2|w|^2v + 2w|v|^2 + w^2\bar{v} + \bar{w}v^2)(t')dt'.\]
Again taking $X_{\mathbb R}^{\frac {2-\al}4,\frac12+}(J)$ norms on
both sides of the above equation and applying \eqref{nonli-est}, we
have
\begin{align*}
\|w\|_{X_{\mathbb R}^{\frac {2-\al}4, \frac12+}(J)} &\lesssim \|u(0) - V(0)\|_{H^{\frac {2 - \al}4}} + \|\eta(t)e\|_{X_{\mathbb R}^{\frac {2-\al}4,\frac12+}}\\
 &\qquad\quad+ \||w|^2w + 2|w|^2v + 2w|v|^2 + w^2\bar{v} + \bar{w}v^2\|_{X_{\mathbb R}^{\frac {2-\al}4,-\frac12+}(J)}\\
&\lesssim \|u(0) - V(0)\|_{H^{\frac {2 - \al}4}} + \|\eta(t)e\|_{X_{\mathbb R}^{\frac {2-\al}4,\frac12+}}\\
&\qquad\quad+ \|w\|_{X_{\mathbb R}^{\frac {2-\al}4,\frac12+}(J)}(\|w\|_{X_{\mathbb R}^{\frac {2-\al}4,\frac12+}(J)} + \|V\|_{X_{\mathbb R}^{\frac {2-\al}4,\frac12+}(J)})^2.
\end{align*}
If $\ep$ is sufficiently small, the continuity argument with respect
to time gives the desired bound.
\end{proof}
\begin{lem}\label{error}
Let $e$ be a solution to the initial value problem $ie_t +
(-\Delta)^\frac {\al}2e = E,\;\; e(0) = 0$,  and let $\eta$ be the
smooth time cut-off function given in Lemma \ref{pertub}.   Then
\[\|\eta(t)e\|_{X_{\mathbb R}^{\frac {2-\al}4,\frac12+}} \lesssim \ep N^{-\al/2}.\]
\end{lem}
For the proof of this lemma, we make use of the following which is
in \cite{cct}.

\begin{lem}[Lemma 2.1 \cite{cct}] \label{inho-sob}
Let $-\frac 12 < s, \sigma > 0$, and $w \in H^\sigma(\mathbb{R})$.
For $M > 1, \tau > 0, x_0 \in \mathbb{R}$ and $A > 0$ let
\[\widetilde w (x) = Ae^{iMx}w(\frac {x - x_0}{\tau}).\]
\begin{enumerate}
\item Suppose that $s \geq 0$. Then there exists a constant $C_1 < \infty$,
 depending only on $s$, such that
\[\|\widetilde w\|_{H^s} \leq C_1|A|\tau^{1/2} M^s \|w\|_{H^s}\]
for all $w, A, x_0$ whenever $M\cdot\tau \geq 1$.
\item Suppose that $s < 0$ and that $\sigma \geq |s|$.
Then there exists a constant $C_1 < \infty$, depending only on $s$ and $\sigma$, such that
\[\|\widetilde w\|_{H^s} \leq C_1|A|\tau^{1/2}M^s\|w\|_{H^\sigma}\]
for all $w, A, x_0$ whenever $1 \leq \tau \cdot M^{1 + (s/\sigma)}$.
\item There exists $c_1 > 0$ such that for each $w$ there exists $C_w < \infty$ such that
\[\|\widetilde w\|_{H^s} \geq c_1|A|\tau^{1/2}M^s\|w\|_{L^2}\]
whenever $\tau \cdot M \geq C_w.$
\end{enumerate}
\end{lem}
\begin{proof}[Proof of Lemma \ref{error}]
Using \eqref{nonli-est} and Plancherel's theorem, we have
\begin{align*}
\|\eta(t)e\|_{X_{\mathbb R}^{\frac {2 - \al}4, \frac12+}} &\lesssim \|\eta(t)E\|_{X_{\mathbb R}^{\frac {2 - \al}4, -\frac12+}} = \|\langle \xi \rangle^{\frac {2 - \al}4} \langle \tau - |\xi|^\al \rangle^{-\frac12+} \widehat{\eta(t)E}\|_{L^2_{\tau,\xi}}\\
&\leq \|\langle \xi \rangle^{\frac {2 - \al}4} \widehat{\eta(t)E}\|_{L^2_{\tau,\xi}}  = \|\eta(t) \langle \xi \rangle^{\frac {2 - \al}4} \mathcal{F} E\|_{L^2_{t,\xi}}\\
&\leq \|\langle \xi \rangle^{\frac {2 - \al}4} \mathcal{F}E\|_{L^\infty_tL^2_\xi([0,2] \times \mathbb{R})}.
\end{align*}
It suffices to show
\[\sup_{0 \leq t \leq 2} \|E\|_{H^{\frac {2-\al}4}} \lesssim \ep N^{-\al/2}.\]
Since $E=e^{iNx} e^{iN^\al t} R(-i\partial_y) v(s,y)$, by Lemma
\ref{inho-sob} with $M = N, \tau = N^{\frac \al2 - 1}$ we see that
$\|E\|_{H^{\frac {2-\al}4}} \lesssim \|R(-i\partial_y)v\|_{H^{\frac
{2-\al}4}}$. Recalling that $R(\xi)$ is given by \eqref{remain}, it
suffices to show $\|R(-i\partial_y)v\|_{H^{\frac {2-\al}4}} \lesssim
\ep N^{-\al/2}$. Since $\|v\|_{H^{5}} \lesssim \ep$ by Theorem
\ref{ill-NLS}, we need only to show that
\begin{align}\label{error est}
\big|R(\xi)\big| \leq cN^{-\al/2}|\xi|^3\;\;\mbox{for all}\;\;N \gg 1.
\end{align}

Let $c_1 = \max \big(8\al (\frac {\al(\al - 1)}{2})^{- \frac 32},
\frac {2^{4 - \al}}{6} (2 - \al)(\frac {\al(\al - 1)}{2})^{-\frac
12}\big)$,\\ $c_2 = \max \big(8\al (\frac {\al(\al - 1)}{2})^{- \frac 32}, \frac {\big(\frac {\al(\al - 1)}{2}\big)^\frac 12}{\al} \big)$ and let $$f(\xi) = \big|(\frac {\al(\al - 1)}{2} N^{\al -
2})^{-\frac 12}\xi + N\big|^\al,\;\; P(\xi) =  N^\al + \frac {\al
N^{\al - 1}}{(\frac {\al(\al - 1)}{2}N^{\al - 2})^\frac 12}\xi +
\xi^2, $$ so that  $R(\xi)=f(\xi) - P(\xi)$. We also denote
$g(\xi)=-c_1N^{-\al/2}\xi^3 + P(\xi)$, $h(\xi) = c_2N^{-\frac \al2}\xi^3 + P(\xi)$. Then it suffices to show $\big|R(\xi)\big| \leq c_1N^{-\al/2}|\xi|^3$ on $\xi > \xi_1$ and $f \leq g, h \leq f$ on $\xi \leq \xi_1$ for some $\xi_1 < 0$. The following are easy to
check:
\begin{align*}
f'(\xi) &= \al \big|\big(\frac {\al(\al - 1)}{2} N^{\al - 2}\big)^{-\frac 12}\xi + N\big|^{\al - 2}\Big(\big(\frac {\al(\al - 1)}{2} N^{\al - 2}\big)^{-\frac 12}\xi + N\Big)\\
&\qquad \times \big(\frac {\al(\al - 1)}2 N^{\al - 2}\big)^{-\frac 12},\\
f''(\xi) &= 2\big|(\frac {\al(\al - 1)}2 N^{\al - 2})^{-\frac 12}\xi + N\big|^{\al - 2}N^{2 - \al},\\
f'''(\xi) &= 2(\al - 2)\big|(\frac {\al(\al - 1)}{2} N^{\al - 2})^{-\frac 12}\xi + N\big|^{\al - 4}\big((\frac {\al(\al - 1)}{2} N^{\al - 2})^{-\frac 12}\xi + N\big)\\
&\qquad \times N^{2 - \al} \big(\frac {\al(\al - 1)}{2} N^{\al - 2}\big)^{-\frac 12},\\
g'(\xi) &= -3c_1N^{-\frac \al2} \xi^2 + 2\xi + \frac {\al N^{\al - 1}}{(\frac {\al(\al - 1)}2 N^{\al - 2})^\frac 12},\quad g''(\xi) = -6c_1N^{-\frac \al2}\xi + 2,\\
h'(\xi) &= 3c_2N^{-\frac \al2}\xi^2 + 2\xi + \frac {\al N^{\al - 1}}{(\frac {\al(\al - 1)}2 N^{\al - 2})^\frac 12} > \frac 23 \times \frac {\al N^{\al - 1}}{(\frac {\al(\al - 1)}2 N^{\al - 2})^\frac 12},
\end{align*}
provided that the derivatives exist.

Let us set  $\xi_1 = -\frac 12 (\frac {\al(\al - 1)}{2} N^{\al -
2})^\frac 12 N$ and $\xi_2 = -2 (\frac {\al(\al - 1)}{2} N^{\al -
2})^\frac 12 N$. Then we consider separately three cases $\xi \ge
\xi_1;\;\; \xi_2 \leq \xi < \xi_1;\;\; \xi < \xi_2$. If $\xi \ge
\xi_1$, $f$ is three times differentiable and
$$|f'''(\xi)| \le | f'''(\xi_1)| = (2-\al)\Big(\frac {\al(\al - 1)}2\Big)^{- \frac 12} 2^{4 - \al} N^{-\al/2}.$$ Hence by Taylor's theorem, we get \eqref{error est}.
We need only to handle the remain two cases.

For both cases it is easy to show $h(\xi) \leq f(\xi).$  In fact,
observe that $h(\xi_1) \leq \big(\frac {\al(\al-1)}{8} + 1 - \frac
{3\al}{2} \big)N^\al \leq (\frac 12)^\al N^\al = f(\xi_1)$. Since
$f'$ is increasing,  $f'(\xi) \leq f'(\xi_1) = (\frac 12)^{\al -1
}\frac {\al N^{\al - 1}}{(\frac {\al(\al - 1)}2 N^{\al - 2})^\frac
12} \leq h'(\xi)$ for  $\xi \leq \xi_1$. Hence, $h(\xi) \leq f(\xi)$
if $\xi \leq \xi_1$.

To show that $f(\xi) \leq g(\xi)$ for $\xi_2 \leq \xi < \xi_1$,
observe that $ f(\xi_1) = \big(\frac N2\big)^\al  \le \big(\frac
\al2 + \frac {\al(\al - 1)}8 + 1\big)N^\al \leq g(\xi_1).$ Hence, it
suffices to show $f'(\xi) \geq g'(\xi)$ for $\xi_2 \leq \xi <
\xi_1$. Since $f'$ is increasing, $ f'(\xi) \geq f'(\xi_2) =
-\al\big(\frac {\al(\al - 1)}2\big)^{- \frac 12}N^{\frac \al 2}. $
Since $g'$ is increasing, $ g'(\xi) \leq g'(\xi_1) \leq -5\al
\big(\frac {\al(\al - 1)}2\big)^{-\frac 12}N^{\frac \al 2}. $ Hence,
$f'(\xi) \geq g'(\xi)$ for $\xi_2 \leq \xi < \xi_1$.

Finally, we show $f(\xi) \le g(\xi)$ for $\xi < \xi_2$. We note that
$f''(\xi) \le g''(\xi)$ and
$$
f(\xi_2) = N^\al \le (64\al + 2\al(\al - 1) - 2\al + 1)N^\al \le g(\xi_2).
$$
Since $f'(\xi_2) = -\al(\frac {\al(\al - 1)}{2})N^{\frac{\al}2} \geq
(-96 + 2\al)N^{\frac \al2} \ge g'(\xi_2)$, $f'(\xi) \geq g'(\xi)$.
This together with $f(\xi_2) \le g(\xi_2)$ gives $f(\xi) \le
g(\xi)$.
\end{proof}

Now we are ready to prove  Theorem \ref{ill}.
\begin{proof}[Proof of Theorem \ref{ill}]
Let $0 < \delta \ll \ep \ll 1$ and $T > 0$ be given. From Theorem
\ref{ill-NLS} we have two global solution $v_1, v_2$ with initial
data $\phi_1, \phi_2$, respectively, such that
\begin{align}\label{ill-1-NLS}
\|\phi_1\|_{H^s},\|\phi_2\|_{H^s} \lesssim \ep,
\end{align}
\begin{align}\label{ill-2-NLS}
\|\phi_1 - \phi_2\|_{H^s} \lesssim \delta,
\end{align}
\begin{align}\label{ill-3-NLS}
\sup_{0 \leq t \leq T}\|v_1(t) - v_2(t)\|_{H^s} \gtrsim \ep,
\end{align}
\begin{align}\label{ill-4-NLS}
\sup_{0 \leq t \leq \infty}\|v_1(t)\|_{H^5}, \|v_2(t)\|_{H^5} \lesssim \ep.
\end{align}
Define $V_1, V_2$ by
\begin{align}\label{appsol}
V_j(t,x) := e^{iNx} e^{iN^\al t} v_j(s,y),\;\; j= 1, 2,
\end{align}
where $(s,y)$ is given by \eqref{changevar}. And let $u_1, u_2$ be
smooth global solutions of \eqref{eq} with initial data $V_1(0,x),
V_2(0,x)$, respectively.

Now we rescale these solutions to have the conditions \eqref{ill-1},
\eqref{ill-2} and \eqref{ill-3} satisfied. Let $\lambda \gg 1$ be a
large parameter to be chosen later. For $j=1,2$, set
\[u_j^\lambda := \lambda u_j(\lambda^\al t, \lambda x),\,\,\, V_j^\lambda := \lambda V_j(\lambda^\al t, \lambda x).\]
 Thus we have
\[u_j^\lambda(0,x) = V_j^\lambda(0,x) = \lambda e^{iN\lambda x} e^{iN^\al t} v_j\Big(0, \frac {\lambda x + \al N^{\al - 1}t}{(\frac {\al(\al - 1)}2 N^{\al - 2})^\frac 12}\Big).\]
Lemma \ref{inho-sob} with $M = N\lambda, \tau = N^{\frac {\al - 2}2}
\lambda^{-1}$ implies that if $s \ge 0$,
\[\|u_j^\lambda(0)\|_{H^s} \lesssim \lambda^{s + 1/2} N^{s - (2 - \al)/4} \|v_j(0)\|_{H^1};\]
if $\frac {2 - 3\al}{4(\al + 1)} < s <0$,
\[\|u_j^\lambda(0)\|_{H^s} \lesssim \lambda^{s + 1/2} N^{s - (2 - \al)/4} \|v_j(0)\|_{H^1}.\]
We choose $\lambda = N^{((2 - \al)/4 - s)/(s + 1/2)}$. By
\eqref{ill-1-NLS} and \eqref{ill-2-NLS} we have
\[\|u_j^\lambda(0)\|_{H^s} \lesssim \ep, \,\,\,\|u_1^\lambda(0) - u_2^\lambda(0) \|_{H^s} \lesssim \delta.\]

Now we show \eqref{ill-3}. Rescaling gives
\begin{align*}
\|u_j^\lambda(t) - V_j^\lambda(t)\|_{H^s} &\lesssim \lambda^{\max(s,0) + 1/2}
\|u_j(\lambda^\al t) - V_j(\lambda^\al t)\|_{H^s}\\
 &\leq \lambda^{\max(s,0) + 1/2}\|u_j(\lambda^\al t) - V_j(\lambda^\al t)\|_{H^{(2-\al)/4}}.
 \end{align*}
Lemma \ref{pertub} and induction argument on time interval up to $\log N/\lambda^\al$ yield
\[\|u_j(\lambda^\al t) - V_j(\lambda^\al t)\|_{H^{(2-\al)/4}} \lesssim \ep N^{-\al/2 + \eta},\]
whenever $0 < t \ll \log N/\lambda^\al$. Hence we have
\[\|u_j^\lambda(t) - V_j^\lambda(t)\|_{H^s} \lesssim
\lambda^{\max(s,0) + 1/2}\ep N^{-\al/2 + \eta}.\] From the hypothesis
$\frac {2 - 3\al}{4(\al + 1)} < s$ it follows that, for a
sufficiently small $\eta>0$,
\[\|u_j^\lambda(t) - V_j^\lambda(t)\|_{H^s} \ll \ep.\]
Applying Lemma \ref{inho-sob} with $M = N\lambda, \tau = N^{\frac {\al - 2}2} \lambda^{-1}$, we have
\[\|u_j^\lambda(t)\|_{H^s} \leq \|u_j^\lambda(t) - V_j^\lambda(t)\|_{H^s} +  \|V_j^\lambda(t)\|_{H^s} \lesssim \ep + \|v_j(\lambda^\al t)\|_{H^s} \lesssim \ep.\]
From \eqref{ill-3-NLS}, we can find a time $t_0 > 0$ such that
$\|v_1(t_0) - v(t_0)\|_{L^2} \gtrsim \ep.$ Fixing $t_0$, we may
choose $N$ so large that $t_0 \ll \log N$. From \eqref{ill-4-NLS}
and Lemma \ref{inho-sob}, we get
\[\|V_1(t_0/\lambda^\al) - V_2(t_0/\lambda^\al)\|_{H^s} \sim \ep.\]
Choosing $N$ large enough, we can make $t_0/\lambda^\al < T$.
Therefore \eqref{ill-3} follows.
\end{proof}

\subsection*{Acknowledgments} Y. Cho was supported by the Research Funds of Chonbuk National University
2014, G. Hwang supported by NRF grant
2012R1A1A1015116, 2012R1A1B3001167 (Republic of Korea), S. Kwon  partially supported by NRF grant 2010-0024017
(Republic of Korea), and S. Lee  supported by NRF grant 2009-0083521
(Republic of Korea).


\begin{thebibliography}{00}





\bibitem{bst} N. Burq, P. Gerard, and N. Tzvetkov, {\it An instability property of the nonlinear Schr\"odinger equation on $S^d$}, Math. Res. Lett. 9 (2002), no. 2-3, 323--335.

\bibitem{cct} M. Christ, J. Colliander, T. Tao, {\it
Asymptotics, frequency modulation, and low regularity ill-
posedness for canonical defocusing equations}, Amer. J. Math. \textbf{125} (2003), no. 6, 1235-1293.

\bibitem{cct2} \bysame, {\it Instability of the periodic nonlinear Schr\"odinger equation}, preprint arXiv:math/0311227 (2003).





\bibitem{chho} Y. Cho, H. Hajaiej, G. Hwang, and T. Ozawa, {\it On the Cauchy problem of fractional Schr\"{o}dinger equation with Hartree type nonlinearity}, Funkcialaj Ekvacioj {\bf 56} (2013), 193-224.


\bibitem{chho2} \bysame, {\it On the orbital stability of fractional Schr\"odinger equations}, Comm. Pure Appl. Anal. \textbf{13} (2014), 1267-1282.


\bibitem{chkl1} \bysame, {\it Profile decompositions and blowup phenomena of mass critical fractional Schr\"odinger equations}, Nolinear Analysis \textbf{86} (2013), 12-29.



\bibitem{cholee} Y. Cho and S. Lee, {\it Strichartz estimates in spherical
coordinates}, Indiana Univ. Math. J. 62 (2013), no. 3, 991-1020.



\bibitem{cox} Y. Cho, T. Ozawa, S. Xia, {\it Remarks on some dispersive
estimates}, Commun. Pure Appl. Anal., {\bf 10} (2011), no. 4,
1121-1128.


\bibitem{det} S. Demirbas, M. B. Erdo\u{g}an and N. Tzirakis, {\it Existence and uniqueness theory for the fractional Schr\"odinger equation on the torus}, preprint arXiv:1312.5249.







\bibitem{gh} B, Guo and Z. Huo, {\it Global Well-Posedness for the Fractional Nonlinear Schrodinger Equation}, Comm. Partial Differential Equations \textbf{36} (2010),  247-255.






\bibitem{iopu}  A. D. Ionescu and F. Pusateri, {\it Nolinear fractional Schr\"odinger equations in one dimension}, J. Func. Anal. \textbf{266} (2014), 139-176.


\bibitem{kwon} S. Kwon, {\it Well-posedness and ill-posedness of the fifth-order modified KdV equations}, Elec. J. Diff. Eqns, 2008(2008), no.1, 1--15.








\bibitem{la1}  N, Laskin, {\it Fractional quantum mechanics and L\'{e}vy path
integrals}, Phys. Lett. A, \textbf{268} (2000), 298-305.






\bibitem{mol} L. Molinet, {\it On ill-posedness for the one-dimensional periodic cubic Schr\"odinger equation}, Math. Res. Lett. \textbf{16} (2009), 111-120.






\bibitem{tao} T. Tao, {\it Multilinear weighted convolution of $L^2$ functions, and applications to nonlinear dispersive equations}, Amer. J. Math. 123 (2001), no. 5, 839-908.


\end{thebibliography}
\end{document}